\documentclass{amsart}

\usepackage{amsmath,url,graphicx,amssymb,enumerate,stmaryrd, amsthm}
\usepackage[pagebackref,colorlinks,citecolor=blue,linkcolor=blue,urlcolor=blue,filecolor=blue]{hyperref}
\usepackage{microtype}
\usepackage[all]{xy}
\usepackage{dsfont}
\usepackage{aliascnt}
\usepackage[vcentermath]{youngtab}
\usepackage{pinlabel}
\usepackage{etoolbox}
\usepackage{enumitem}
\usepackage{tikz} \usetikzlibrary{matrix,arrows}

\usepackage[hmargin=3cm,vmargin=3cm]{geometry}

\newcommand{\m}{\to}

\providecommand{\Star}{\text{Star}}
\providecommand{\Link}{\text{Link}}
\providecommand{\Cone}{\text{Cone}}

\providecommand{\genus}{\text{genus}}
\providecommand{\valence}{ \text{valence}}

\DeclareMathOperator{\Aut}{Aut}

\DeclareMathOperator{\rank}{rank}

\newcommand{\CV}{\mathbb{A}}

\numberwithin{thmcounter}{section}
\newaliascnt{thmauto}{thmcounter}

\newaliascnt{Defauto}{thmcounter}

\newaliascnt{exauto}{thmcounter}

\newaliascnt{quauto}{thmcounter}

\newaliascnt{lemauto}{thmcounter}

\newaliascnt{propauto}{thmcounter}

\newaliascnt{corauto}{thmcounter}

\newaliascnt{remauto}{thmcounter}

\newaliascnt{convauto}{thmcounter}

\newtheorem*{ThmA'}{Theorem A'}
\newtheorem*{ThmB'}{Theorem B'}
\newtheorem*{ThmC'}{Theorem C'}

\newtheorem{theorem}[thmauto]{Theorem}
\newtheorem{lemma}[lemauto]{Lemma}
\newtheorem{proposition}[propauto]{Proposition}
\newtheorem{corollary}[corauto]{Corollary}
\theoremstyle{definition}
\newtheorem{definition}[Defauto]{Definition}

\newtheorem{remark}[remauto]{Remark}
\newtheorem{convention}[convauto]{Convention}

\title{A degree theorem for the simplicial closure of Auter Space}

\author{Juliet Aygun}
\address{Department of Mathematics, Cornell University, USA}
\email{ja742@cornell.edu}
\author{Jeremy Miller}
\address{Department of Mathematics, Purdue University, USA}
\email{jeremykmiller@purdue.edu}
\thanks{Jeremy Miller was supported in part by Simons Foundation Collaboration Grant 41000749 and NSF Grant DMS 2202943.}

 \date{\today}

\begin{document}

\begin{abstract}
The degree of a based graph is the number of essential non-basepoint vertices after generic perturbation. Hatcher--Vogtmann's degree theorem states that the subcomplex of Auter Space of graphs of degree at most $d$ is $(d-1)$-connected. We extend the definition of degree to the simplicial closure of Auter Space and prove a version of Hatcher--Vogtmann's result in this context.
\end{abstract}
\maketitle


\section{Introduction}


\emph{Auter Space}, denoted $\CV_n$, is a version of Teichm\"uller space for automorphism groups of free groups. \emph{Outer Space}, the analogous construction for Outer automorphism groups, was introduced by Culler-Vogtmann \cite{CulV}. Auter Space is a space of connected based marked metric graphs.   Metric means the edges are equipped with a length (which sums to $1$) and marked means we fix the data of an isomorphism between the fundamental group of the graph and the free group on $n$ letters. Vertices of valence $1$ or $2$ are not allowed. This space has been instrumental for many calculations of the homology of $\text{Aut}(F_n)$ and its subgroups (see e.g. \cite{HVrat,Assembling}). To study homological stability properties of $\text{Aut}(F_n)$, Hatcher--Vogtmann \cite{HatcherVogtmannCerf} proved a ``degree theorem'' for Auter Space. The degree of a based graph $\Gamma \in \CV_n$ is defined to be $$\deg(\Gamma) : =\sum_{v \neq v_0} ( \valence(v) -2 )$$ where the sum is taken over all vertices of $\Gamma$ other than the basepoint vertex $v_0$. Qualitatively, $\deg(\Gamma)$ is the number of non-basepoint vertices of a generic perturbation of $\Gamma$ fixing a neighborhood of the basepoint. Generic in this context implies the valence of every non-basepoint vertex will be $3$. For example, the graph in the center of Figure \ref{Picture} has degree $2$ since a perturbation of it is the graph on the right of Figure \ref{Picture} which has $2$ non-basepoint vertices.

Let $\CV_n^{\leq d}$ denote the subcomplex of $\CV_n$ of graphs of degree $\leq d$. Hatcher--Vogtmann's degree theorem states that $\CV_n^{\leq d}$ is highly connected.

\begin{theorem}[Hatcher--Vogtmann]
 $\CV_n^{\leq d}$ is $(d-1)$-connected.
\end{theorem}

\begin{figure}[!ht] \begin{center}\scalebox{.75}{\includegraphics{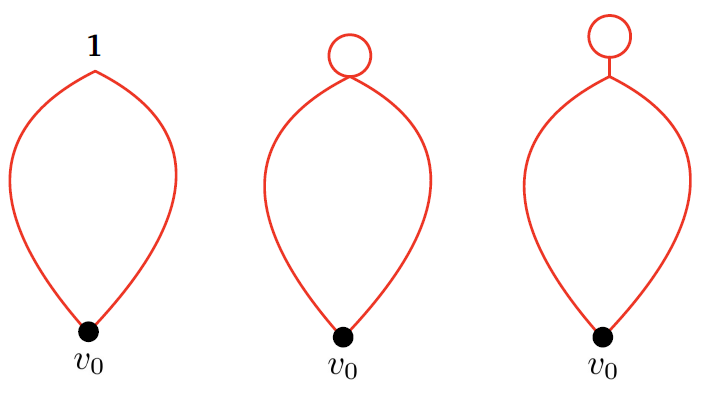}}\end{center}
\label{Picture}
\caption{Graphs of degree $2$}
\end{figure}  

Our main result is a version of this theorem for the simplicial closure of Auter Space. In this introduction, we will describe a heuristic model of the simplicial closure of Auter Space as a space of graphs. However, in the body of the paper, we will use a model for the simplicial closure of Auter Space in terms of a simplicial complex of sphere systems in a certain $3$-manifold (see Definition \ref{spheresystemdef}). The connection between spaces of graphs and complexes of sphere systems dates back to work of Hatcher \cite[Appendix]{HatcherStab}. See Br\"uck \cite[Section 4.7.1]{Bthesis} for a discussion of sphere system models of the simplicial closure of Outer Space.

 Recall that each edge in a graph in $\CV_n$ has a length, and the sum of the lengths is required to be equal to $1$. If the edge is not a loop, then the length of the edge is allowed to converge to $0$ and the associated edge in the graph is collapsed. The simplicial closure of Auter Space, denoted $\overline{\CV}_n$, is a partial compactification where loops are allowed to collapse as well. Vertices are labeled with a natural number (genus) indicating how many loops have been shrunk (as well as extra data relating to the marking). For example, as the top loop in the middle graph in Figure \ref{Picture} shrinks, it converges to the labeled graph on the left in the simplicial closure. The genus of a vertex should be though of as the number of infinitesimal loops based at the vertex. Vertices of valence $1$ or $2$ are now allowed if they have positive genus. 

 We define the degree of a graph $\Gamma$ in $\overline{\CV}_n$ to be $$\deg(\Gamma) : =\sum_{v \neq v_0} ( \valence(v)+2 \cdot \genus(v) -2 ).$$ We can still think of $\deg(\Gamma)$ as the number of non-basepoint vertices of a generic perturbation of $\Gamma$. For example, a perturbation of the graph on left of Figure \ref{Picture} is the graph on the right which has $2$ non-basepoint vertices. The factor of $2$ on the genus term in the sum reflects the fact that loops contribute twice to the valence and the genus of a vertex is viewed as the number of infinitesimal loops based at that vertex.

Analogously, we define $\overline{\CV}_n^{\leq d}$ to be the subcomplex of $\overline{\CV}_n$ consisting of graphs with degree $\leq d$. Our main result is the following which can be thought of as a ``degree theorem'' for the simplicial closure of Auter Space. 

\begin{theorem}\label{MainAuterSpace}
 $\overline{\CV}_n^{\leq d}$ is $(n+d-2)$-connected.
\end{theorem}

	Work of Hatcher \cite{HatcherStab} gives an alternative model of  ${\CV}_n$ as a complex of $2$-sphere systems in a certain $3$-manifold. This $3$-manifold perspective also gives a convenient model for $\overline{\CV}_n$. See Definition \ref{Sd} for a rigorous definition of the spaces  $\overline{\CV}_n^{\leq d}$. The $d=0$ case of Theorem \ref{MainAuterSpace}
is due to Hatcher--Vogtmann \cite[Theorem 2.5]{HVfreeFactors} as $\overline{\CV}_n^{\leq 0}$ is isomorphic to a complex of simplexwise non-separating sphere systems. 

\begin{remark}
Hatcher--Vogtman proved their degree theorem to study homological stability properties of $\Aut(F_n)$ \cite{HatcherVogtmannCerf}. We do not know of a direct application our degree theorem for the simplicial closure to homological stability. One possible motivation for Theorem \ref{MainAuterSpace} is the following. The complex  $\overline{\CV}_n^{\leq 1}$ is very similar to a complex used by Himes, the second author, Nariman, and Putman  \cite{HMNP} to study the first homology group of $\Aut(F_n)$ with coefficients in the top homology of Hatcher--Vogtmann's complex of free factors \cite{HVfreeFactors} $FC_n$. More specifically,  Himes, the second author, Nariman, and Putman used high connectivity of a variant of $\overline{\CV}_n^{\leq 1}$ to show $H_1(\Aut(F_n); \widetilde H_{n-2}(FC_n)) \cong 0$ for $n \geq 2$. Possibly, high connectivity of the complexes $\overline{\CV}_n^{\leq d}$ could be used to show $H_d(\Aut(F_n); \widetilde H_{n-2}(FC_n))$ vanishes for  $n$ sufficiently large compared with $d$. 
\end{remark}

\subsection*{Acknowledgments}
We thank Benjamin Br\"uck, Sam Nariman, Peter Patzt, Andrew Putman, and Robin Sroka for helpful conversations.

\section{Background}
In this section, we define certain complexes of sphere systems in 3-manifolds. We also discuss known connectivity results. 

\subsection{Sphere complexes}

We are interested in $2$-spheres in the following type of $3$-manifolds.

\begin{definition}
Let $M_{n,b}$ be the manifold $\#_n S^1 \times S^2$ with $b$ disjoint, open 3-balls removed.
\end{definition}
 
Without explicitly saying, we will often pick disjoint representatives of spheres and no property in this paper depends on this choice. We will use the term \emph{genus} of $M_{n,b}$ or a graph to mean the rank of its fundamental group. We will always take $b \geq 1$ and fix a basepoint $p_0 \in \partial M_{n,b}$.

\begin{definition}
Let $S(M_{n,b})$ be the simplicial complex whose vertices are isotopy classes of embedded 2-spheres not isotopic to a boundary component or bounding a 3-ball. A collection of $p+1$ vertices $\{S_0,...,S_p\}$ forms a $p$-simplex $[S_0,...,S_p]$ if the spheres can be isotoped to be disjoint.
\end{definition}

The following subcomplex of $S(M_{n,b})$ imposes a simplexwise non-separating condition. 

\begin{definition} \label{spheresystemdef}
Let $Y(M_{n,b})$ be the subcomplex of $S(M_{n,b})$ where $[S_0,...,S_p]$ is a $p$-simplex of $Y(M_{n,b})$ if $M_{n,b} \backslash \cup_i S_i$ is connected. 
\end{definition}

Given a simplex $\sigma=[S_0,...,S_p] $ of $S(M_{n,b})$, $M_{n,b} \backslash \cup_i S_i $ is a disjoint union of connected manifolds which are interiors of manifolds with boundary. We will often denote the associated manifold with boundary by $M \backslash \sigma$.

Any simplex of $S(M_{n,b})$ has a corresponding \emph{dual graph}. 

\begin{definition}
Let $\sigma = [S_0,...,S_p]$ be a simplex of $S(M_{n,b})$. Define the dual graph of $\sigma$, denoted $\Gamma(\sigma)$, to be the graph with vertices connected components of $M_{n,b} \backslash \sigma$ and edges $\{S_0,...,S_p\}$. The edge $S_i$ is attached to the not necessarily distinct vertices corresponding to connected components it is adjacent to in $M_{n,b} \backslash \sigma$. We label each vertex $v$ with the genus of the associated connected component and denote this number by $\genus(v)$.
\end{definition}

See Figure \ref{CorePic} for examples of simplices of $S(M_{n,b})$ and their associated dual graphs. 

\begin{remark}
The association of a simplex of $S(M_{n,1})$ to its dual graph gives a correspondence between $S(M_{n,1})$ and the space $\overline{\CV}_n$ heuristically described in the introduction. The simplicial coordinates in $S(M_{n,1})$ giving edge lengths in $\overline{\CV}_n$ (see Hatcher \cite[Appendix]{HatcherStab} and Br\"uck \cite[Section 4.7.1]{Bthesis}).
\end{remark}

Note that if $\sigma$ is a face of $\tau$, then $\Gamma(\sigma)$ is obtained from $\Gamma(\tau)$ by collapsing the edges corresponding to the vertices of $\tau \backslash \sigma$. 

An edge is called a \emph{self-loop} if it is attached to a single vertex. Let $\tau=[S_0,\ldots,S_p]$ and $\sigma=[S_0,\ldots,\hat S_i,\ldots,S_p]$. The edge $S_i$ in $\Gamma(\tau)$ corresponds to a self-loop if and only if $M_{n,b} \backslash \sigma$ and  $M_{n,b} \backslash \tau$ have the same number of connected components. 

A simplex $\sigma$ of $S(M_{n,b})$ is a simplex of $Y(M_{n,b})$ if and only if $\Gamma(\sigma)$ is a rose (a graph with a single vertex).

Define the genus of $M_{n,b} \backslash \sigma$ to be the sum of the genera of all of the connected components. The following appears in Hatcher--Vogtmann  \cite[Proof of Theorem 2.5]{HVfreeFactors}.

\begin{proposition} \label{rank}
Given $\sigma$ a simplex of $S(M_{n,b}),$ if $g$ is the genus of $M_{n,b} \backslash \sigma$, then $$g = n - \rank(\pi_1(\Gamma(\sigma))).$$ 
\end{proposition}


The \emph{valence} of a vertex $v$ in $\Gamma(\sigma)$ is the number of half-edges out of $v$. For example, non-self-loops contribute $1$ to the valence of $2$ vertices and self-loops contribute $2$ to the valence of $1$ vertex. 

We now will state the definition of degree in the context of sphere systems. 

\begin{definition} Let $\sigma$ be a simplex of $S(M_{n,1})$.
The \emph{degree} of $\sigma$ is $$\deg(\sigma) = \sum_{v \neq v_0, \, v \in \Gamma(\sigma)} (\valence(v) + 2 \cdot \genus(v)- 2) $$ where $v_0$ denotes the vertex corresponding to the basepoint connected component.
\end{definition}

\begin{remark}
We only define degree for simplices in $S(M_{n,b})$ for $b=1$ as it is unclear if our definition is the correct notion for a higher number of boundary components. We will need the complexes $S(M_{n,b})$ and $Y(M_{n,b})$ for $b>1$, but our main theorem only concerns the case $b=1$. The only time we use the assumption that $b=1$ directly is in Lemma \ref{removeVertex}, although this lemma is key for most of what follows.
\end{remark}

\begin{definition} \label{Sd}
Let $S^{\leq d}(M_{n,1})$ be the subcomplex of $S(M_{n,1})$ consisting of simplices $\sigma = [S_0,...,S_p]$ of $S(M_{n,1})$ with $\deg(\sigma) \leq d$. 
\end{definition}

Our model for the space $\overline{\CV}_n^{\leq d}$ described in the introduction is $S^{\leq d}(M_{n,1})$. 

Corollary \ref{complex} verifies that $S^{\leq d}(M_{n,1})$ is a simplicial complex as opposed to just a collection of simplices. Note that $S^{\leq 0}(M_{n,1})$ is isomorphic to  $Y(M_{n,1})$ as both complexes consist entirely of sphere systems whose associated dual graphs have no non-basepoint vertices (roses). This uses the fact that every non-basepoint vertex in a dual graph must have valence at least $3$ or be associated to a component with positive genus.

\subsection{Known connectivity results}
We now review some connectivity results about sphere systems that will be used to prove our main theorem. 

\begin{theorem} \cite[Theorem 2.1]{HatcherStab} \label{Scontract}
$S(M_{n,b})$ is contractible for all $n,b \geq 1$.
\end{theorem}

\begin{theorem} \cite[Theorem 2.5]{HVfreeFactors} \label{Yconn}
$Y(M_{n,b})$ is $(n-2)$-connected for all $n \geq 0$ and $b \geq 1$.
\end{theorem}

The following is Hatcher--Wahl \cite[Theorem 3.1 (1)]{HWboundary} in the case $C = \emptyset$ and $k=0$. Note that the complex Hatcher--Wahl consider involves spheres and disks, but no disks are allowed when $C = \emptyset$. 

\begin{theorem}\cite[Theorem 3.1 (1)]{HWboundary}
$S(M_{0,b})$ is $(b-5)$-connected for all $b \geq 1$. \label{SM0}
\end{theorem}

In this paper, we denote the simplicial join by $\star.$

\begin{proposition} \label{join}
Let $A_1, \ldots, A_q$ be simplicial complexes. If each $A_i$ is $n_i$-connected, then $A_1 \star ... \star A_q$ is $(n_1 + . . . +n_q + 2q - 2)$-connected. 
\end{proposition}

\section{Properties of degree}

In this section, we establish basic properties of degree. We begin by introducing the notion of a \emph{pillar} of a sphere system.

\begin{definition}
The pillar of a sphere system $\sigma$ in $S(M_{n,b})$ is the face (or empty simplex) consisting of vertices corresponding to edges in $\Gamma(\sigma)$ that go between the basepoint vertex and a non-basepoint vertex. We say that a sphere system is a pillar if it equals its pillar.

\end{definition}

Non-separating sphere systems have empty pillars while separating sphere systems have non-empty pillars. The pillar of the simplex depicted in the top of Figure \ref{CorePic} is depicted in the bottom of Figure \ref{CorePic}.

\begin{figure}[!ht] 

\begin{center}\scalebox{.5}{\includegraphics{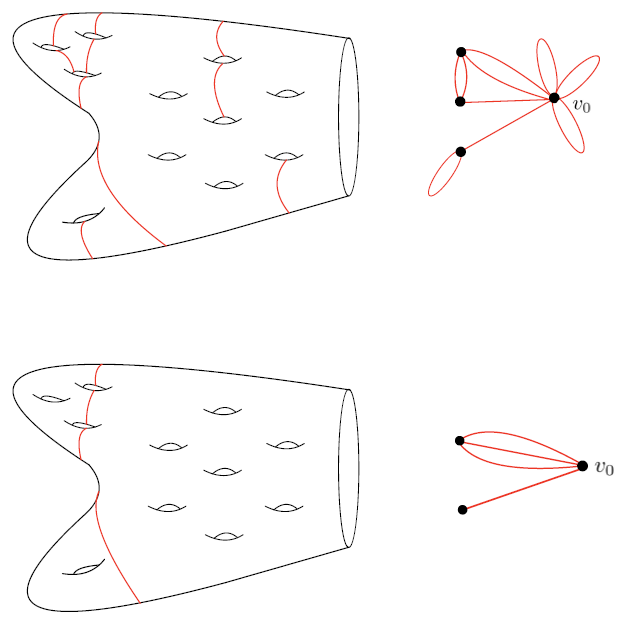}}

\end{center}

\caption{Pillar of a sphere system and associated dual graphs} \label{CorePic} 
\end{figure}

\begin{lemma} \label{removeVertex}
Let $\tau$ be a simplex of $S(M_{n,1})$ and let $\sigma$ be the face of $\tau$ corresponding to removing a vertex $S$. If $S$ is part of the pillar of $\tau$, then $\deg(\sigma) < \deg(\tau).$  Otherwise, $\deg(\sigma) = \deg(\tau).$ 
\end{lemma}

\begin{proof}
We first consider the case where $S$ is part of the pillar. Then $S$ corresponds to an edge between $v_0$ and some other vertex $v$. In this case, $\Gamma(\sigma)$ is obtained from $\Gamma(\tau)$ by collapsing an edge between $v_0$ and $v$. Thus, $$\deg(\tau) - \deg(\sigma)=\valence(v) +2 \cdot \genus(v)-2$$ since the terms coming from non-basepoint vertices other than $v$ appear twice and cancel. 

Pick $x$ and $y$ such that the connected component of $M_{n,1} \backslash \tau$ associated to $v$ is isomorphic to $M_{x,y}$. Then $y=\valence(v)$ since $v$ is not the basepoint of $\Gamma(\tau)$ and $S(M_{n,1})$ has no non-basepoint boundary components. We also have that $x=\genus(v)$. Different spheres in sphere systems are not allowed to be isotopic so $(x,y) \neq (0,2)$. Spheres do not bound balls so $(x,y) \neq (0,1)$. Thus either $\valence(v) \geq 3$ or $\valence(v) \geq 1$ and $\genus(v) \geq 1$. Hence $\valence(v) +2 \cdot \genus(v)-2>0$ and so $\deg(\sigma) < \deg(\tau)$ if $S$ is part of the pillar.

Now assume that $S$ is not part of the pillar. There are $3$ cases to consider.

\begin{enumerate}

\item $S$ corresponds to a self-loop at $v_0$:

In this case, $\Gamma(\sigma)$ and $\Gamma(\tau)$ agree away from the basepoint. The sum defining degree is only over non-basepoint vertices, and hence $\deg(\sigma)=\deg(\tau)$. 

\item $S$ corresponds to a self-loop at $v$ and $v \neq v_0$:

Collapsing the edge corresponding to $S$ in $\Gamma(\tau)$ decreases the valence of $v$ by $2$ but increases the genus by $1$, leaving degree unchanged.

\item $S$ corresponds to an edge between $v_1$ and $v_2$ and $v_0,v_1,v_2$ are all distinct:

Let $v$ denote the vertex of $\Gamma(\sigma)$ corresponding to the collapsed edge associated to $S$. We have that $$\deg(\tau) - \deg(\sigma)=$$ $$\Big ( \valence(v) +2 \cdot \genus(v)-2 \Big )- \Big ( \valence(v_1) +2 \cdot \genus(v_1)-2 \Big ) -\Big ( \valence(v_2) +2 \cdot \genus(v_2)-2 \Big ) $$ since the graphs only differ at $v$, $v_1$, and $v_2$ and hence the other terms cancel in the sums. This quantity vanishes since $$\valence(v)=\valence(v_1)+\valence(v_2)-2$$ $$\text{and }\genus(v)=\genus(v_1)+\genus(v_2).$$ Thus, $\deg(\sigma)=\deg(\tau)$.
\end{enumerate} \end{proof}

Lemma \ref{removeVertex} gives the following immediately.

\begin{corollary} \label{complex}
If $\sigma$ is a face of $\tau$ in $S(M_{n,1})$, then $\deg(\sigma) \leq \deg(\tau)$. In particular, $S^{\leq d}(M_{n,1})$ is a well-defined simplicial complex.

\end{corollary}

We now conclude that the degree of a sphere system only depends on its pillar. 

\begin{corollary} \label{degEqpillar}
Let $c$ be the pillar of a sphere system $\tau$. Then $\deg(c)=\deg(\tau)$.
\end{corollary}

\begin{proof} This follows from applying Lemma \ref{removeVertex} to every vertex of $\tau$ that is not a vertex of $c$.
\end{proof}

\section{New connectivity results}

In this section, we prove $S^{\leq d}(M_{n,1})$ is $(n+d-2)$-connected (Theorem \ref{mainthm1}).

\begin{convention}
Given a simplex $\sigma$ of $S(M_{n,b})$, we denote the basepoint connected component of $M_{n,b} \backslash \sigma$ by $O^\sigma$ and denote the other connected components by $N_1^\sigma, \ldots, N_k^\sigma$.
\end{convention}

If $c$ is a pillar, all spheres have $O^c$ on one side and some $N_i^c$ on the other side.

There is a natural inclusion $$S(O^\sigma) \star S(N_1^\sigma) \ldots \star S(N_k^\sigma) \hookrightarrow S(M_{n,b}).$$ We will often conflate elements of the domain of this inclusion with their image.

\begin{lemma} \label{badnessmaps}
Let $X$ be a simplicial complex, $f:X \to S(M_{n,1})$ a simplicial map, $c$ a pillar in the image of $f$, and $\tau \subset X$ some simplex of maximal dimension subject to the condition that $f(\tau)=c$. Assume for all $\omega$ in $X$ that $\deg(f(\omega)) \leq \deg(c)$. Then $$f(\emph{Link}_X(\tau)) \subset Y(O^c) \star S(N^c_1) \star ... \star S(N^c_k).$$
\end{lemma}
\begin{proof}
Let $c$ and $\tau$ be as above. By the disjointness condition defining $S(M_{n,b})$, it follows that $$f(\Link_X(\tau)) \subset S(O^c) \star S(N^c_1) \star ... \star S(N^c_k).$$ Suppose for the purposes of contradiction that there is a simplex $\theta$ in $\Link_X(\tau)$ such that $f(\theta)$ is not a simplex of $$ Y(O^c) \star S(N^c_1) \star ... \star S(N^c_k).$$ Assume $\theta$ has minimal dimension among such simplices. Necessarily, $f(\theta)$ is a simplex of $S(O^c)$ since otherwise it would not be minimal dimensional if some of its vertices mapped to other connected components. We must have that $f(\partial \theta) \subset Y(O^c)$, or otherwise we could replace $\theta$ by one of its faces and find a smaller dimensional simplex with the desired properties. Let $x$ be a vertex of $\theta$ and let $\theta'$ be the face opposite from $x$. Since $f(x)$ corresponds to an edge between the basepoint and another vertex in $\Gamma(f(\theta) \star c)$, Lemma \ref{removeVertex} implies $$\deg(f(\theta') \star c ) < \deg(f(\theta) \star c ).$$ Since $f(\theta') \subset Y(O^c)$, $f(\theta') \star c$ has $c$ as its pillar. Corollary \ref{degEqpillar} implies $$\deg(f(\theta') \star c))=\deg(c).$$ Thus, $\deg( f(\theta) \star c) > \deg(c)$ which contradicts our assumption that $\deg(f(\omega)) \leq \deg(c)$ for all simplices $\omega$ of $X$.\end{proof}

Similarly, for any pillar $c$, every simplex in $Y(O^c) \star S(N^c_1) \star ... \star S(N^c_k) \star \partial c$ must have lower degree than $c$.

\begin{lemma} \label{lowdeg}
If $c$ is a pillar, then $\deg(c) > \deg(\sigma)$ for all $$\sigma \subset Y(O^c) \star S(N^c_1) \star ... \star S(N^c_k) \star \partial c.$$
\end{lemma}
\begin{proof}
Suppose $\sigma \subset Y(O^c) \star S(N^c_1) \star ... \star S(N^c_k) \star \partial c$ and let $c_{\sigma}$ denote the pillar of $\sigma$. Let $\sigma'$ be the simplex obtained from $\sigma$ by adding in all of the vertices of $c$. Since $\sigma'$ has pillar equal to $c$,  Corollary \ref{degEqpillar} implies that $\deg(\sigma')=\deg(c)$. Since $\sigma'$ is obtained from $\sigma$ by adding vertices corresponding to edges between the basepoint and another connected component in $\Gamma(\sigma')$, Lemma \ref{removeVertex} implies $\deg(\sigma') > \deg(\sigma)$.
\end{proof}

We also have a useful result about the connectivity of $Y(O^c) \star S(N^c_1) \star ... \star S(N^c_k)$.

\begin{lemma} \label{pillarlinkconnect}

Let $c$ be a simplex of $S(M_{n,1})$ which is a pillar. Then $Y(O^{c}) \star S(N^{c}_1) \star ... \star S(N^{c}_k)$ is $(n+\deg(c)-\dim(c)-3)$-connected.
\end{lemma}
\begin{proof}
If at least one $N^c_i$ has nonzero genus, then Theorem \ref{Scontract} implies $Y(O^{c}) \star S(N^{c}_1) \star ... \star S(N^{c}_k)$ is contractible because the simplicial join of a contractible complex with any other complex is contractible. Thus, it remains to address the case that each $N^c_i$ has genus zero. Let $b_i$ denote the number of boundary components of $N^c_i$ and let $g$ denote the genus of $O^c$. By Theorem \ref{Yconn}, Theorem \ref{SM0}, and Proposition \ref{join}, $$Y(O^{c}) \star S(N^{c}_1) \star ... \star S(N^{c}_k)$$ is $\Big((g-2)+(b_1 -5)+...+(b_k -5) + 2(k+1) -2\Big)$-connected. Because $\sum_{i=1}^k b_i = \dim(c)+1,$ after simplification, we see that the join is $\Big(g+\dim(c)+1-3k-2 \Big)$-connected. 

Since each $N_i^c$ has genus $0$, we have that $$\deg(c)= \sum_{v \neq v_0} (\valence(v) -2) = \dim(c)+1-2k.$$ By Proposition \ref{rank}, $g=n - \rank(\pi_1(\Gamma(c)))$. Using that the Euler characteristic can be computed using ranks of homology or number of cells, we see that the rank of the fundamental group of a connected graph is $1$ more than the number of edges minus the number of vertices. Thus, $$\rank(\pi_1(\Gamma(c))) = \dim(c)+1-k.$$ Substituting all of this in yields the result.
\end{proof}


\begin{definition}
Any simplicial complex structure on a $0$-dimensional manifold is called combinatorial. A simplicial complex structure on a $d$-manifold is called combinatorial if links of $p$-simplices are combinatorial $(d-p-1)$-spheres.
\end{definition}

Combinatorial simplicial complex structures on manifolds with boundary are defined analogously. We are now ready to prove the degree theorem for $S(M_{n,1})$.

\begin{theorem} \label{mainthm1}
$S^{\leq d}(M_{n,1})$ is $(n+d-2)$-connected.
\end{theorem}
\begin{proof}
Fix $i \leq n+d-2$. By the simplicial approximation theorem, it is sufficient to show that any map $f:S^{i} \to S^{\leq d}(M_{n,1})$ which is simplicial with respect to some combinatorial simplicial structure on $S^i$ can be extended to a map $D^{i+1} \to S^{\leq d}(M_{n,1})$. Note that $S(M_{n,1})$ is contractible, so we can extend $f$ to $\hat{f}: D^{i+1} \to S(M_{n,1})$ for some combinatorical simplicial structure on $D^{i+1}$. The fact that we require the simplicial structure on $D^{i+1}$ to be combinatorical ensures that the links of interior simplices will be homeomorphic to spheres of the appropriate dimension.

Let $x$ be the maximum degree of a simplex in the image of $\hat f$ and let $y$ be the maximal dimension of a simplex of $D^{i+1}$ whose image is a pillar of degree $x$. Corollary \ref{degEqpillar} states that the degree of a simplex agrees with the degree of its pillar. If $x \leq d$, then we are done as this would mean that $\hat f$ factors through $S^{\leq d}(M_{n,1}),$ and hence that $f:S^{i} \to S^{\leq d}(M_{n,1})$ is null homotopic. We may therefore assume $x>d$. Fix $\omega$ a simplex of $D^{i+1}$ of dimension $y$ with $\hat f(\omega)=c$ and $\deg(c)=x$. We will modify the simplicial structure on $D^{i+1}$ in the interior of $\Star_{D^{i+1}}(c)$ so that we remove the simplex $\omega$ and all new simplices $\upsilon$ that we introduce have the property that if $\hat f(\upsilon)$ is a pillar, then either $\deg(\hat f(\upsilon)) < x$ or $\deg(\hat f(\upsilon)) = x$ and $\dim \upsilon <y$. By Lemma \ref{badnessmaps}, $\Link_{D^{i+1}}(\omega)$ maps to $$Y(O^c) \star S(N^c_1) \star ... \star S(N^c_k).$$ By Lemma \ref{pillarlinkconnect}, this simplicial join is $(n+x-\dim(c)-3)$-connected and hence $(n+x-y-3)$-connected. Since the simplicial structure is combinatorial, $\Link_{D^{i+1}}(\omega) \cong S^{i-y}$. Since $i \leq n+d-2$ and $d<x$, we have that $i-y \leq n+x-y-3$. Thus, we can extend $\hat{f}|_{\Link_{D^{i+1}}(\omega)} $ to a map $$g: \Cone\Big(\Link_{D^{i+1}}(\omega)\Big) \m Y(O^c) \star S(N^c_1) \star ... \star S(N^c_k).$$ Modify the simplicial structure on $D^{i+1}$ by replacing $\Star_{D^{i+1}}(\omega)$ with $\Cone(\Link_{D^{i+1}}(\omega)) \star \partial \omega$ and change the definition of $\hat f$ so that it agrees with $g$ on the new vertices. If $\dim c=\dim \omega=y$, then Lemma \ref{lowdeg} implies that simplices in $\Cone(\Link_{D^{i+1}}(\omega)) \star \partial w$ will map to simplices of degree $<x$. If $\dim c < \dim \omega=y$, then all simplices in $\Cone(\Link_{D^{i+1}}(\omega))$ that map to pillars will have dimensions $<y$. 

In other words, we have lowered the degree or made the map more injective on pillars. Since dimensions cannot be negative, iterating this procedure will modify the map $\hat f$ so that its image has no pillars of degree $\geq x$. We continue this process until the image of $\hat f$ lies in $S^{\leq d}(M_{n,1})$.
\end{proof}

We view Theorem \ref{mainthm1} as the rigorous version of Theorem \ref{MainAuterSpace}.

\bibliographystyle{amsalpha}
\bibliography{AutFn}

\end{document}